\documentclass[12pt, one sided]{amsart}
\RequirePackage[colorlinks,citecolor=blue,urlcolor=blue]{hyperref}

\usepackage{latexsym,amssymb,amsmath,amsfonts,amsthm}
\usepackage{amsthm,amsmath}
\usepackage{float}
\usepackage{enumerate}
\usepackage[margin=3cm]{geometry}
\usepackage{bbm}
\usepackage{subfigure}
\usepackage{graphicx}
\usepackage[toc,page]{appendix}
\usepackage{multirow}
\usepackage{amsfonts}
\usepackage[all]{xy}
\usepackage{caption}
\usepackage{geometry}                
\usepackage{graphicx}
\usepackage{amssymb}
\usepackage{epstopdf}
\usepackage{color}
\usepackage{bbm, dsfont}
\usepackage{hyperref}
\usepackage{cleveref}
\usepackage[utf8]{inputenc}
\usepackage{amssymb,amsthm,amsmath,amssymb,wrapfig,dsfont}
\usepackage[dvipsnames]{xcolor}
\definecolor{myred}{RGB}{251,154,133}
\definecolor{myblue}{RGB}{153,206,227}
\definecolor{mylightblue}{RGB}{0, 150, 255}
\definecolor{mygreen}{RGB}{32, 210, 64}
\definecolor{mygray}{RGB}{220, 220, 220}

\usepackage{tikz}
\usetikzlibrary{decorations.pathmorphing}
\tikzset{snake it/.style={decorate, decoration=snake}}
\usetikzlibrary{shapes.geometric,positioning,decorations.pathreplacing} 

\usepackage{breqn}
\usepackage{upgreek}
\usepackage{enumitem}  

\newtheorem{theorem}{Theorem}

\newtheorem{lemma}{Lemma}[section]

\newtheorem{corollary}[lemma]{Corollary}

\def\beq{ \begin{equation} }
	\def\eeq{ \end{equation} }

\def\square{\vcenter{\vbox{\hrule height .4pt
			\hbox{\vrule width .4pt height 5pt \kern 5pt
				\vrule width .4pt} \hrule height .4pt}}}

\newcommand{\BC}{{\mathbb{C}}}

\newcommand{\BH}{{\mathbb{H}}}

\newcommand{\BN}{{\mathbb{N}}}

\newcommand{\BR}{{\mathbb{R}}}

\newcommand{\ind}{{\mathbbm{1}}}

\newcommand{\prob}{{\bf P}}

\newcommand{\bae}{\begin{equation}\begin{aligned}}
		\newcommand{\eae}{\end{aligned}\end{equation}}
\newcommand{\ev}{\mathbf{E}}

\DeclareFontFamily{OML}{rsfs}{\skewchar\font'177}
\DeclareFontShape{OML}{rsfs}{m}{n}{ <5> <6> rsfs5 <7> <8> <9>
	rsfs7 <10> <10.95> <12> <14.4> <17.28> <20.74> <24.88> rsfs10 }{}
\DeclareMathAlphabet{\mathfs}{OML}{rsfs}{m}{n}

\newcommand{\ppp}{\mathfs{A}}

\newcommand{\ima}{\text{Im}}
\newcommand{\rea}{\text{Re}}

\newcommand{\slit}{\varphi}

\newcounter{relctr} 
\everydisplay\expandafter{\the\everydisplay\setcounter{relctr}{0}} 

\begin{document}
	
	\title{Logarithmic fluctuations of Stationary Hastings-Levitov}

\author{Noam Berger}
\address[Noam Berger]{Department of mathematics, TUM}
\urladdr{https://www.math.cit.tum.de/en/probability/people/berger/}
\email{noam.berger@tum.de }

\author{Eviatar B. Procaccia}
\address[Eviatar B. Procaccia]{Technion - Israel Institute of Technology}
\urladdr{http://procaccia.net.technion.ac.il}
\email{eviatarp@technion.ac.il}

\maketitle
\begin{abstract}
We prove that the fluctuation field $\{M_t(x)\}_{x\in\BR}$ of stationary Hastings-Levitov$(0)$ exhibits logarithmic spatial correlations. Moreover, by studying the infinitesimal generator of the imaginary part of $M_t(0)$, we show that for some $\beta>0$, $\max_{x\in[0,t]}\ima M_t(x)<\beta\log t$ with high probability, as $t\to\infty$.
\end{abstract}
	

\section{Introduction}
Stationary Hastings-Levitov$(0)$ (SHL$(0)$) was introduced and studied by the authors with Turner as a stationary off-lattice version of Diffusion Limited Aggregation defined as the composition of conformal slit maps from the upper half complex plane. Unlike the classical Hastings-Levitov process on the complement of the unit disk \cite{hastings1998laplacian, norris2012hastings, silvestri2017fluctuation}, where slit normalization is required to avoid blowup, in the half plane geometry un-normalized particle sizes remain tight. Using this simplicity, it is shown in \cite{berger2022growth} that the harmonic measure of an interval is a martingale that scales to $0$ like the distance of two brownian motions starting at the interval's end points. It follows that at time $t>0$, the typical trees in the aggregate reach a height of $\frac{\pi}{2}t$ and contain an order of $t^{3/2}$ particles. This matches a prediction by Meakin \cite{meakin1983diffusion} and was verified in computer simulations \cite{procaccia2021dimension} for DLA grown on the half plane (the so called Stationary DLA \cite{procaccia2019stationary,procaccia2020stationary,mu2022scaling}). Other stationary versions of aggregation processes where studied, and all share some universal geometric attributes such as finiteness of all trees a.s at time $\infty$ \cite{berger2014stretched,antunovic2017stationary}.

The SHL$(0)$ is defined by the following procedure: considering an intensity $1$ Poisson point process $\ppp$ on $\BR\times\BR_+$, and order $\ppp\cap([-m,m]\times \BR_+):=\{(x_1,t_1),(x_2,t_2),\ldots\}$ by the second coordinate (arrival times). Define for any $z\in\BH=\{z\in\BC:\ima z>0\}$,
\bae
\tilde{F}_t^m(z)=\left\{\begin{array}{ll}
0 & t<t_1\\
\slit_{x_1}\circ \slit_{x_2}\circ\cdots\circ \slit_{x_k}(z)& t_k\le t<t_{k+1} 
\end{array}\right.
,\eae
with the slit map $\slit_x(z)=x+\sqrt{(z-x)^2-1}$, using the branch of $\sqrt{\cdot}$ for which the imaginary part is positive.
The SHL$(0)$ is defined by taking the limit $\tilde{F}_t(z)=\lim_{m\to\infty}\tilde{F}_t^m(z)$. 
\begin{figure}[h!]
\centering
\includegraphics[ height=1.5in, width=1\textwidth]{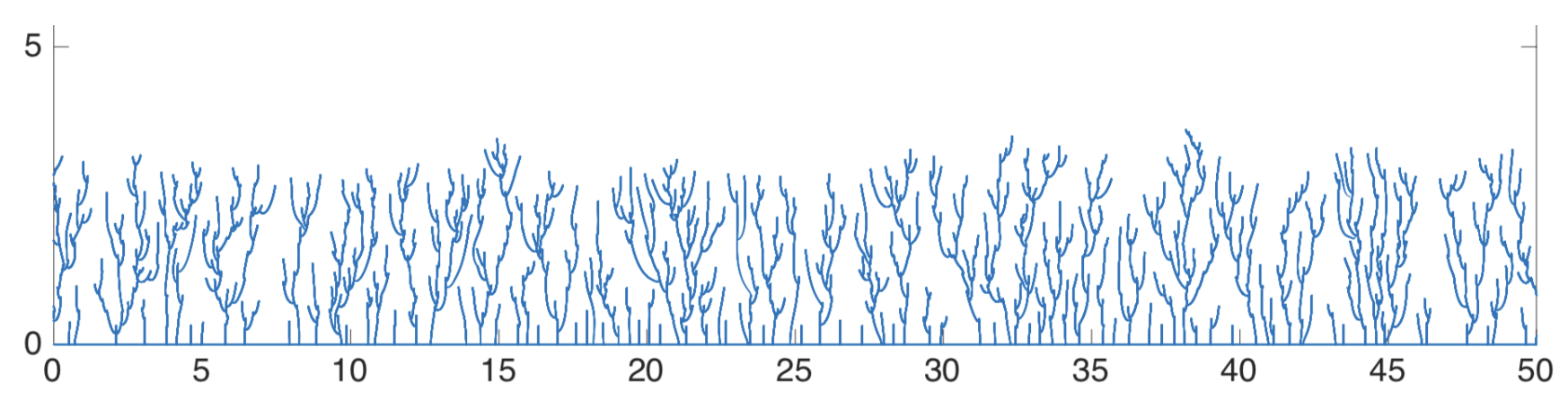}
\caption{Computer simulation of $\{\tilde{F}_t(x)\}_{x\in\BR}$ \label{fig:processsim}}
\end{figure}

Note that for any $t$, $\tilde{F}^m_t$ has the same distribution as 
\bae
{F}_t^m(z)=\left\{\begin{array}{ll}
0 & t<t_1\\
\slit_{x_k}\circ \slit_{x_{k-1}}\circ\cdots\circ \slit_{x_1}(z)& t_k\le t<t_{k+1} 
\end{array}\right.
,\eae	
and by taking the limit we obtain the backwards SHL$(0)$, $F_t(z)=\lim_{m\to\infty}F_t^m(z)$. It is shown in \cite{berger2022growth} that one can write 
\bae\label{eq:martingale}
F_t(z)=z+it\frac{\pi}{2}+M_t(z)
,\eae
where $M_t(z)$ is a zero mean martingale, which can be realized as the limit of the martingale part in the Doob decomposition 
\begin{equation}\label{eq:finitewindow}
F^m_t(z)=z+D^m_t(z)+M^m_t(z).
\end{equation} 

In this paper we study the field $\{M_t(x)\}_{x\in\BR}$ and prove it observes logarithmic spatial correlation and study the maximum of the field. 
\begin{figure}[h!]
\centering
\includegraphics[ height=1.5in, width=1\textwidth]{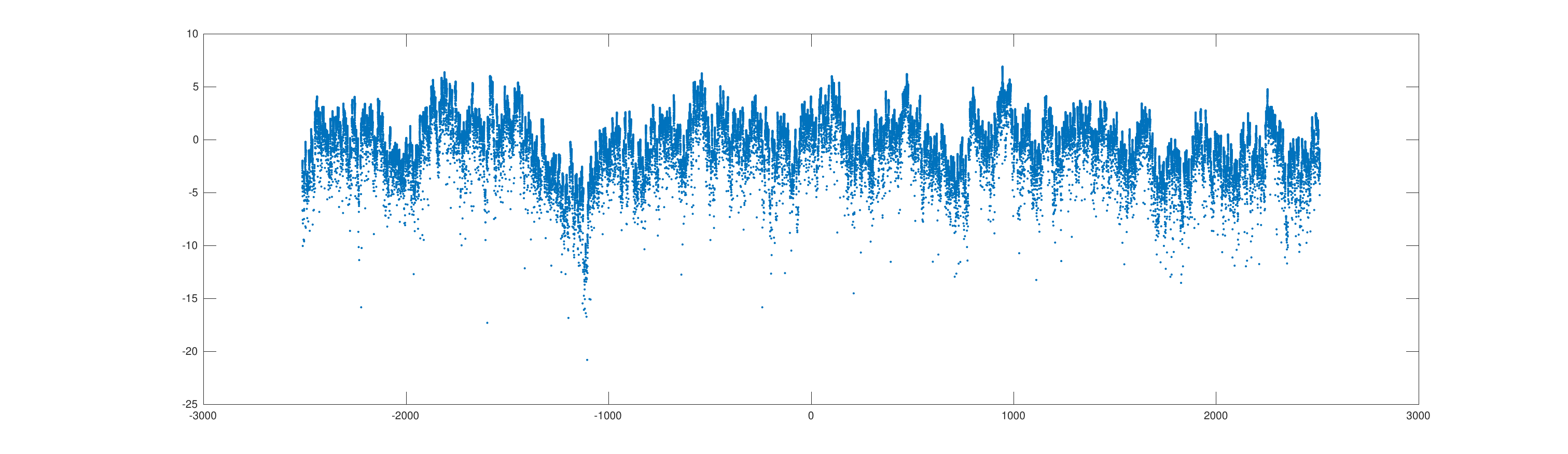}
\caption{Computer simulation of $\{M_t(x)\}_{x\in\BR}$ \label{fig:flucsim}}
\end{figure}

Logarithmic correlation for the small particle limit of the original Hastings-Levitov $(0)$ process were discovered by Silvestri in \cite{silvestri2017fluctuation}. We propose a distinct approach compared to Silvestri’s, focusing on the physical process rather than its scaling limit. Specifically, we demonstrate the logarithmic correlations of fluctuations for the SHL(0) process itself. Note that by computer simulations (see Figure \ref{fig:flucsim}), without considering the scaling limit, it doesn't appear that the field $\{M_t(x)\}_{x\in\BR}$ is Gaussian. See Figure \ref{fig:hist} for a simulated histogram of the marginal distribution, note the lack of symmetry. 
\begin{figure}[h!]
\includegraphics[width=0.5\linewidth]{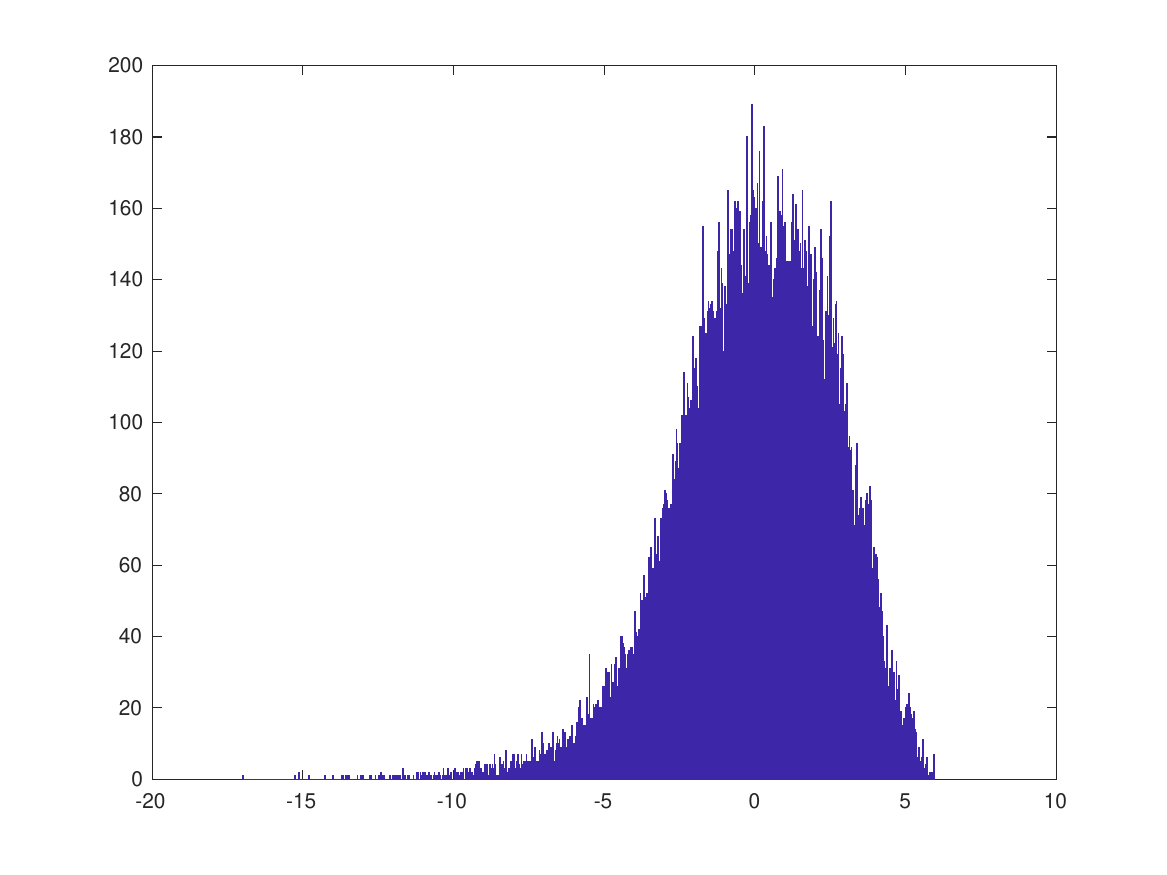} 
\caption{$M_t(0)$ empirical distribution histogram \label{fig:hist}}
\end{figure}
\subsection{Results}
First we improve the diffusive bound proved in \cite{berger2022growth}.
\begin{theorem}\label{thm:var} As $t\to\infty$,
$\ev[|M_t(0)|^2]=  \frac{\pi}{4}\log t(1+o(1)).$
\end{theorem}
Next we claim that $\{M_t(x)\}_x$ admits logarithmic correlations.	
\begin{theorem}\label{thm:cov}
$$\mathrm{Cov}(M_t(0),M_t(b))=\frac{\pi}{4}\min\{\log  (t)-\log  (b), 0\}(1+o(1)),$$
where the little $o$ notation is with respect to $b\to\infty$.
\end{theorem}
	
Lastly for the behavior maximum of the field, let
\bae
\mathfs{M}_t=\max_{x\in[0,t]}\ima M_t(x).
\eae
\begin{theorem}\label{thm:max_fluc}
There exists a $c>0$ such that for any $\beta>0$ large enough, for every $t>0$ large enough,
\bae
\prob\left(\max_{x\in[0,t]}\ima M_t(x)>\beta \log t\right)<\frac{c}{t^{1/2}}.
\eae
\end{theorem}

	
\section{Proof  of Theorem \ref{thm:var}}
First we improve the law of large numbers estimate from \cite{berger2022growth}.
\begin{lemma}\label{lem:llnrate}
There exist  $\xi_{\ref{lem:llnrate}}, c_{\ref{lem:llnrate}}>0$ such that for every $z\in\BH$,
$$\prob\left[\ima F_t(z)<\ima z+\xi_{\ref{lem:llnrate}}t\right]<e^{-c_{\ref{lem:llnrate}}t}.$$
\end{lemma}
\begin{proof}
W.l.o.g we assume that $\frac{t}{m}\in\BN$, for some $m$ that will be chosen later, and we write
$$
F_t(z)=G_{t/m}\circ G_{t/m-1}\circ \cdots\circ G_1(z)
,$$
where the $G_i$ are i.i.d. and distributed as $F_m$. By \cite[Lemma 5.1]{berger2022growth} we obtain that there are $M_m^i$ i.i.d with mean zero such that $G_i(z)=z+i\frac{m\pi}{2}+M_m^i$. Moreover, 
\begin{equation}\label{eq:growthspeed}
\prob\left[\ima G_1(z)>\ima z +\frac{m\pi}{4}\right]\ge 1-\prob\left[|M_1(z)|>\frac{m\pi}{4}\right]\ge1-\frac{c}{(1+\ima z)m}
.\end{equation}
Thus, for $m$ large enough, there is a $\xi>0$ such that for any $z\in \BH$,
\bae
\prob\left[\ima G_1(z)>\ima z +\frac{m\pi}{4}\right]\ge \xi
.\eae
Since $G_i$ are independent, we obtain that
\bae
\prob\left[\ima M_i(G_{i-1}\circ G_{i-2}\circ \cdots\circ G_1(z))>-\frac{m\pi}{4}\right]\ge \xi
.\eae

Denote $N_t=|\{i\le t/m:\ima M_i(G_{i-1}\circ G_{i-2}\circ \cdots\circ G_1(z))>-\frac{m\pi}{4}\}|$. By Hoeffding's inequality there is a $\zeta>0$ such that, 
\bae
\prob\left(N_t<\frac{\xi}{2}t/m\right)<\zeta^{t/m}
.\eae
Under the event $\left\{N_t\ge \frac{\xi}{2}t/m\right\}$, we obtain $F_t(z)\ge \frac{\pi}{4}\frac{\xi}{2}t$. By choosing $\xi_{\ref{lem:llnrate}}=\frac{\pi}{2}\frac{\xi}{2}$ we obtain for some $c_{\ref{lem:llnrate}}>0$ that
$$\prob\left[\ima F_t(z)<\ima z+\xi_{\ref{lem:llnrate}}t\right]<e^{-c_{\ref{lem:llnrate}}t}.$$
\end{proof}
With Lemma \ref{lem:llnrate}, we are ready to prove Theorem \ref{thm:var}. The proof employs a bootstrapping approach, iteratively refining fluctuation bounds and incorporating them back into the martingale representation.
\begin{proof}[Proof  of Theorem \ref{thm:var}] We first condition on the likely event $$H_1(\eta)=\left\{\forall 		s>\eta: \ima F_s(0)>\frac{s}{\log(s^\alpha)}\right\},$$ then bound the fluctuations, in order to guarantee that the event 
$$H_2(s)=\left\{\ima F_s(0)>\pi s/2 -s^{2/3}\right\}$$ is highly likely.	
By Lemma \ref{lem:llnrate}, there exists a $c>0$ such that for $\alpha$ large enough, that will be chosen later, $$\prob[H_1(\eta)]>1-\frac{c}{\eta^{c\alpha-1}}.$$
By \eqref{eq:martingale} and  \cite[Lemma 3.3]{berger2022growth} we write,
\bae
\ev\left[|M_t(0)|^2\right]&=\ev\left[\int_0^t\int_{-\infty}^\infty  \Big|\slit_x(F_s(0))-F_s(0)\Big|^2dxds\right]\le \ev\left[\int_0^t  \frac{c}{1+\ima F_s(0)}ds  \right]\\
&=\ev\left[\int_0^\eta  \frac{c}{1+\ima F_s(0)}ds  \right]+\ev\left[\int_\eta^t  \frac{c}{1+\ima F_s(0)}\left(\ind_{H_1(\eta)}+\ind_{H_1(\eta)^c}\right)ds  \right]\\
&\le c\eta +\ev\left[\int_\eta^t  \frac{c}{1+\ima F_s(0)}ds  \ind_{H_1(\eta)}\right]+ct\eta^{1-c\alpha}\\
&\le  c\eta +c\int_\eta^t\frac{\log s}{s}+ct\eta^{1-c\alpha}= c\eta +ct\eta^{1-c\alpha}+c(\log t)^2-c(\log \eta)^2\\
\eae
Now take $\eta= t^\beta$ for $\beta=\frac{1}{6}$ and choosing $\alpha>\frac{1}{c\beta}$ , to obtain 
$$
\ev\left[|M_t(0)|^2\right]\le ct^\beta(1+o(1)).
$$
We thus obtain that
$$
\prob\left[\ima F_t(0)<\frac{\pi t}{2}- t^{\gamma}\right]\le \prob\left[|M_t(0)|>t^{\gamma}\right]\le  \frac{t^\beta}{t^{2\gamma}}
.$$
Choosing $\gamma=\frac{2}{3}$ yields the bound
$$
\prob[H_2(s)]>1- \frac{c}{s^{1+1/6}}
.$$
Repeating the earlier argument with $H_2(\eta)$ replacing $H_1(\eta)$ and using the quantitative Lemma \ref{lem:slit_int_bound},
\bae\label{eq:upper_constant}
\ev\left[|M_t(0)|^2\right]&\le \ev\left[\int_0^t \min\left\{1,\frac{\pi}{4\ima F_s(0)}+\frac{c}{(\ima F_s(0))^{3}}\right\} (\ind_{H_2(s)}+\ind_{H^c_2(s)})ds \right] \\
&\le 1+ \int_1^t \Bigg(\frac{\pi}{4s}+\frac{c}{s^{3}}+\frac{c}{s^{1+1/6}}\Bigg)ds
.\eae
We conclude that
$$
\ev[|M_t(0)|^2]\le \frac{\pi}{4}\log t(1+o(1))
,$$
proving the upper bound.

As for the lower bound, define
\bae
&H'_2(\eta)=\left\{\forall \eta< s: \ima F_s(0)>\pi s/2 -s^{2/3}\right\}\\
&H_3(\eta)=\left\{\forall \eta< s: \ima F_s(0)<\pi s/2 + s^{2/3}\right\}\\
&H_4(\eta)=H_3(\eta)\cap H'_2(\eta).
\eae
Hence, for $\eta$ large enough, by Lemma \ref{lem:slit_int_bound},

\bae\label{eq:lower_exact_const}
\ev\left[|M_t(0)|^2\right]&\ge \ev\left[\int_0^t\int_{-\infty}^\infty  \Big|\slit_x(F_s(0))-F_s(0)\Big|^2dxds \ind_{H_4(\eta)}\right]\\&\ge \ev\left[\int_\eta^t  \min\left\{1,\frac{\pi}{4\ima F_s(0)}-\frac{c}{(\ima F_s(0))^{3}}\right\}ds  \ind_{H_4(\eta)}\right]\\
&\ge \frac{\pi}{4}\prob[H_4(\eta)]\int_\eta^t\frac{1}{s}= \frac{\pi}{4}(1-\eta^{-1/3})\left[\log t-\log \eta-o\left(\eta^{-1}\right)\right],
\eae
Now taking $\eta=\sqrt{t}$, we obtain the lower bound for the variance.
\end{proof}

\section{Proof of Theorem \ref{thm:cov}}
We prove Theorem \ref{thm:cov} by a sequence of lemmas dealing with different regimes of $t$ and $b$.
Before continuing, we will bootstrap the earlier results to improve the fluctuation estimate of Lemma \ref{lem:llnrate}.

\begin{lemma}\label{lem:bootstrapfluctuationrate} For any $z\in\BH$,
$$
\prob\left(\ima F_t(z)<\frac{\pi}{2}t-\sqrt{t}\right)<e^{-c_{\ref{lem:bootstrapfluctuationrate}} t^{\frac{1}{6}}}
.$$
\end{lemma}
\begin{proof}
Let $G_{t^\alpha}^j(z)=z+i\frac{\pi}{2}t^\alpha+M_{t^\alpha}^j(z)$ be i.i.d. distributed as $F_{t^\alpha}(z)$. By Theorem \ref{thm:var}, for any $z\in\BH$,
\bae
\prob\left(|M_{t^\alpha}^j(z)|>t^\beta\right)\le\frac{c\alpha\log t}{t^{2\beta}}
.\eae
Denote $N_t=\left|\{1\le j\le t^{1-\alpha}: |M_{t^\alpha}^j(G_{t^\alpha}^{j-1}\circ\cdots\circ G_{t^\alpha}^1(z))|<t^\beta\}\right|$. By Hoeffding's inequality 
\bae
\prob\left(N_t<\left(1-\frac{2c\alpha\log t}{t^{2\beta}}\right)t^{1-\alpha}\right)<e^{c t^{1-\alpha}}
\eae
Thus, on an event with probability greater than $1-e^{ct^{1-\alpha}}$, we get that
\bae
\ima F_t(z)&\ge \left(1-\frac{2c\alpha\log t}{t^{2\beta}}\right)t^{1-\alpha}\left(\frac{\pi}{2}t^\alpha-t^\beta\right)= \left(1-\frac{2c\alpha\log t}{t^{2\beta}}\right)\left(\frac{\pi}{2}t-t^{1-\alpha+\beta}\right)\\
&\ge\frac{\pi}{2}t-t^{1-\alpha+\beta}-2c\alpha\log t\frac{\pi}{2}t^{1-2\beta}
.\eae
Choose $\alpha=\frac{5}{6}$ and $\beta=\frac{1}{3}$ to obtain the statement for some constant $c_{\ref{lem:bootstrapfluctuationrate}}>0$.
\end{proof}


Next, we prove that in order to study the fluctuations of $F_t$ it is enough to grow the process on a window of size $f(t)$, which depends on $t$ in a way which we describe below. Remember the definition of $M_t^m(z)$ from \eqref{eq:finitewindow}.
\begin{lemma}\label{eq:approxwindow}
There is a $c>0$ such that for any function $f(t)$ satisfying $\lim_{t\to\infty}\frac{t}{f(t)}=0$, for any $t$ large enough,
\bae
\ev\left[\Big|M_t(0)-M_t^{f(t)}(0)\Big|^2\right]<\frac{c t}{f(t)}.
\eae
\end{lemma}
\begin{proof}
Since both $M_s(0)$ and $M_s^{f(s)}(0)$ are zero mean martingales, so is their difference. Thus we can write
\bae\label{eq:bounding_cut_window}
\ev\left[\Big|M_t(0)-M_t^{f(t)}(0)\Big|^2\right]&\le 8 \ev\left[\int_0^t\int_{-f(t)}^{f(t)}\Big|(\slit_x-\text{id})F_s(0)-(\slit_x-\text{id})F_s^{f(t)}(0)\Big|^2dxds\right]\\
 &+8 \ev\left[\int_0^t\int_{|x|>f(t)}\Big|\slit_x(F_s(0))-F_s(0)\Big|^2dxds\right]
 .\eae 
 By \cite[Lemma 3.3]{berger2022growth} we have that for all $x>2$, 
 \[
 \Big|\slit_x(F_s(0))-F_s(0)\Big|^2\le \frac{c}{x^2+(\ima F_s(0))^2}
 .\]
 Thus,
 \bae
 &\ev\left[\int_0^t\int_{|x|>f(t)}\Big|\slit_x(F_s(0))-F_s(0)\Big|^2dxds\right] \\&\le\ev\left[\int_0^t\frac{2}{\ima F_s(0)}\left[\frac{\pi}{2}-\tan^{-1}\left(\frac{f(t)}{\ima F_s(0)}\right)\right]ds\right]\le\frac{2t}{f(t)}
 .\eae
 As for the other summand of \eqref{eq:bounding_cut_window}, We take a path $\gamma^s:[0,1]\to\BC$ connecting  $F_s(0)$ and $F_s^{f(t)}$ and write 
 \bae
 \Big|(\slit_x-\text{id})F_s(0)-(\slit_x-\text{id})F_s^{f(t)}(0)\Big|^2\le |F_s(0)-F_s^{f(t)}(0)|^2\int_0^1\Big|\slit'_x(\gamma^s(u))-1\Big|^2du
 \eae
Denote $g(t)=\ev\left[\max_{s\le t}\Big|F_s(0)-F_s^{f(s)}(0)\Big|^2\right]$, and the event $H_{\ref{lem:bootstrapfluctuationrate}}(s)=\{\ima F_s(z)\ge\frac{\pi}{2}s-\sqrt{s}\}$, thus 
\bae
&\ev\left[\int_0^t\int_{-\infty}^\infty\Big|(\slit_x-\text{id})F_s(0)-(\slit_x-\text{id})F_s^{f(t)}(0)\Big|^2dxds\right]\\
&\le c\int_0^t\prob(H_{\ref{lem:bootstrapfluctuationrate}}(s)^c)+\ev\left[\int_0^t\int_{-\infty}^\infty\Big|(\slit_x-\text{id})F_s(0)-(\slit_x-\text{id})F_s^{f(t)}(0)\Big|^2dxds\ind_{H_{\ref{lem:bootstrapfluctuationrate}}(s)}\right]\\
&\le c+\ev\left[\int_0^t \max_{s\le t}\Big|F_s(0)-F_s^{f(s)}(0)\Big|^2\frac{c}{1+\ima (F_s(0))^3} ds \ind_{H_{\ref{lem:bootstrapfluctuationrate}}(s)}\right]\\
&\le c+\int_0^t g(s)\frac{c}{1+s^3}ds,
\eae
where in the first inequality we used the fact that there is a $c>0$ such that $|\slit_x(z)-z|<c$ for all $x\in \BR$ and $z\in \BH$ (see \cite[Appendix A]{berger2022growth}).
Now by Gronwall's inequality, for some sonstant $c>0$,
\bae
g(t)\le \frac{2t}{f(t)}e^c,
\eae
Finishing the proof.

\end{proof}


\begin{lemma}\label{lem:cov_far}
For any $t>0$ and and $ b\ge 2t\log (t)^3$, we have  
$$\Big|\mathrm{Cov}(M_t(0),M_t(b))\Big|\le c\sqrt{\frac{t}{b}\log t}.$$
\end{lemma}
\begin{proof}
\bae
&\ev[M_t(0)M_t(b)]\\
&=\ev\Big[(M_t(0)-M_t^{f(t)}(0)+M_t^{f(t)}(0))(M_t(b)-M_t^{f(t)}(b)+M_t^{f(t)}(b))\Big]\\
&=\ev\Big[(M_t(0)-M_t^{f(t)}(0))M_t^{f(t)}(b)\Big]+\ev\Big[(M_t(b)-M_t^{f(t)}(b))M_t^{f(t)}(0)\Big]
\\&+\ev\Big[(M_t(0)-M_t^{f(t)}(0))(M_t(b)-M_t^{f(t)}(b))\Big]+\ev[M_t^{f(t)}(0)M_t^{f(t)}(b)]
\eae
By applying Cauchy-Schwartz and Lemma \ref{eq:approxwindow} with the choice $f(t)=b/2>t\log(t)^3$, and the fact that $M_t^{f(t)}(0)$ and $M_t^{f(t)}(b)$ are independent, we obtain
\bae\label{eq:covzerorate}
&\Big|\ev[M_t(0)M_t(b)]\Big|\le 2\sqrt{\frac{ct}{f(t)}\log t}+\frac{ct}{f(t)}+0
\eae
Thus by our choice of $f(t)$ then the RHS of \eqref{eq:covzerorate} is smaller than $c\sqrt{\frac{t}{b}\log t}$. 
\end{proof}


\begin{lemma}
For any $t>b$, we have  
$$\Big|\mathrm{Cov}(M_t(0),M_t(b))\Big|=\frac{\pi}{4}\Big(\log(t)-\log(b)\Big)+O\Big(\log\log(b)+\sqrt{\log\log (b)}\sqrt{\log b}\Big),$$
as $b\to\infty$.
\end{lemma}
\begin{proof}
First we bound using Lemma \ref{lem:cov_far},
\bae\label{eq:3bound}
\ev\Big[(M_{b/\log(b)^3}(b)-M_0(b))(M_{b/\log(b)^3}(0)-M_0(0))\Big]&=\ev\Big[M_{b/\log(b)^3}(b)M_{b/\log(b)^3}(0)\Big]\\&
\le \frac{1}{\log{(b/\log{(b)^3}})}
\eae
Second by the martingale property for $t>s$, $$\ev[|M_t(b)-M_s(b)|^2]=\ev[|M_t(b)|^2-|M_s(b)|^2],$$
so by Theorem \ref{thm:var} 
\bae\label{eq:2bound}
&\ev\Big[(M_{b\log(b)^3}(b)-M_{b/\log(b)^3}(b))M_t(0)\Big] \\ 
&=\ev\Big[(M_{b\log(b)^3}(b)-M_{b/\log(b)^3}(b))(M_t(0)-M_{b\log(b)^3}(0)+M_{b\log(b)^3}(0))\Big] \\ 
&=\ev\Big[(M_{b\log(b)^3}(b)-M_{b/\log(b)^3}(b))M_{b\log(b)^3}(0)\Big] \\ 
&\le\sqrt{\ev\Big[\Big|(M_{b\log(b)^3}(b)-M_{b/\log(b)^3}(b))\Big|^2\Big]\ev\Big[\Big|(M_{b\log(b)^3}(0)\Big|^2\Big]}\\
&\le c\sqrt{\log({b\log(b)^3)}}\Big(\sqrt{\log( b\log (b)^3)-log(b/\log (b)^3)}\Big)\le c\sqrt{\log\log (b)^3}\sqrt{\log b}
,\eae
where in the second equality we used the martingale property and the fact that $M_s(b)$ and $M_s(0)$ are martingales on a common filtration i.e. for $t_1<t_2<t_3$, $\ev\left[(M_{t_2}(b)-M_{t_1}(b))(M_{t_3}(0)-M_{t_2}(0))\right]=0$. 

The third and final estimate we need before we can complete the proof of the lemma follows Cauchy-Schwarz and \eqref{eq:upper_constant},
\bae\label{eq:1bound}
&\left|\ev\Big[(M_{t}(b)-M_{b\log(b)^3}(b))(M_{t}(0)-M_{b\log(b)^3}(0))\Big]\right|\\
&\le \sqrt{\ev\Big[|M_{t}(b)-M_{b\log(b)^3}(b)|^2\Big]}\sqrt{\ev\Big[|M_{t}(0)-M_{b\log(b)^3}(0)|^2\Big]}\\
&\le \frac{\pi}{4}\Big(\log(t)-\log(b\log(b)^3)\Big).
\eae 
Finally we can estimate using \eqref{eq:3bound}, \eqref{eq:2bound}, \eqref{eq:1bound}, 
\bae\label{eq:covequality}
&\Big|\mathrm{Cov}(M_t(0),M_t(b))\Big|=\\
&\Big|\ev\Big[\Big((M_t(b)-M_{b\log(b)^3}(b))+(M_{b\log(b)^3}(b)-M_{b/\log(b)^3}(b))+(M_{b/\log(b)^3}(b)-M_0(b))\Big)\\
&\Big((M_t(0)-M_{b\log(b)^3}(0))+(M_{b\log(b)^3}(0)-M_{b/\log(b)^3}(0))+(M_{b/\log(b)^3}(0)-M_0(0))\Big)\Big]\Big|\\
&\le \frac{\pi}{4}\Big(\log(t)-\log(b)\Big)+O(\log(\log(b))+\sqrt{\log\log (b)}\sqrt{\log b}).
\eae

For the lower bound we write from the first line of \eqref{eq:1bound},
\bae\label{eq:lowerbound}
&\ev\Big[(M_{t}(b)-M_{b\log(b)^3}(b))(M_{t}(0)-M_{b\log(b)^3}(0))\Big]=\\
&\ev\Big[(M_{t}(b)-M_{b\log(b)^3}(b))^2\\&+(M_{t}(b)-M_{b\log(b)^3}(b))(M_{t}(0)-M_{b\log(b)^3}(0)-M_{t}(b)+M_{b\log(b)^3}(b))\Big]\
\eae 
Note that for $r<t$
\bae\label{eq:derivativeuse}
&\ev\left[|M_t(0)-M_t(b)|^2-|M_r(0)-M_r(b)|^2\right]\\
&=\ev\left[\int_r^t\int_{-\infty}^\infty \Big|(\slit_x-\text{id})F_s(0)-(\slit_x-\text{id})F_s(b)\Big|^2dxds\right]\\
&\le \ev\left[\int_r^t\int_{-\infty}^\infty \Big(|F_s(0)-F_s(b)|^2\int_0^1|\slit_x^'(\gamma(u))-1|^2du\Big)dxds\right]\\
&\le \ev\left[\sup_{r\le s\le t}|F_s(0)-F_s(b)|^2 \int_0^1 \int_b^t \int_{-\infty}^\infty |\slit_x^'(\gamma(u))-1|^2dxdsdu\right]\\
&\le c (\log (t)+b^2) \left(\frac{1}{1+r^2}-\frac{1}{1+t^2}\right)
,\eae
Thus, using \eqref{eq:lower_exact_const}, \eqref{eq:lowerbound} and \eqref{eq:derivativeuse} with $r=b\log(b)^3$,
\bae\label{eq:lowerbound2}
&\ev\Big[\Big|(M_{t}(b)-M_{b\log(b)^3}(b))(M_{t}(0)-M_{b\log(b)^3}(0))\Big|\Big]\\
&\ge \frac{\pi}{4}\Big(\log(t)-\log(b\log(b)^3)\Big)\\
 &-\sqrt{\Big(c\Big(\log(t)-\log(b\log(b)^3)\Big)\Big)(\log (t)+b^2) \left(\frac{1}{1+(b\log(b)^3)^2}-\frac{1}{1+t^2}\right)}\\
&=\frac{\pi}{4}\Big(1-O\Big(\frac{1}{(\log b)^6}\Big)\Big)\Big(\log(t)-\log(b)\Big)-O(\log\log b)
\eae
\end{proof}

\section{Bounding the SHL generator}
In this section we wish to study the exponential moments of $M_t(0)$, and later use them to bound the maximal fluctuation in an interval. 
Let $Y_t= \ima (F_t(0))$ and
$$
\Delta(x,\zeta):= \ima \left(\slit_x(i \zeta)-i \zeta\right)
.$$
\begin{theorem}\label{thm:moment}
There is a $c>0$ such that for all $\alpha>0$ and all $z\in\BH$,
$$
\ev\left[e^{\alpha \ima M_t(z)}\right]\le e^{\frac{\pi}{2}\alpha^2e^\alpha }  t^{\alpha^2}
.$$ 
\end{theorem}
\begin{proof}
Note that since the Poisson point process is invariant under independent translations, the imaginary part process $Y_t$ is Markov. In fact, it follows the axiomatic definition of the process \cite[Definition 2.3]{berger2022growth}, that $Y_t$ is a Feller process. Thus, one can write the generator of $Y_t$ as: 
\[
\mathfs{L} f(\zeta)=\int_{-\infty}^{\infty}\left[f(\zeta+\Delta(\zeta,x))-f(\zeta)\right]dx
.\]
We now apply the generator for the function $f(y)=e^{\alpha y}$. Noting that $\max_x\Delta(\zeta,x)=\Delta(\zeta,0)\le\frac{1}{1+\zeta}$,
\bae
\mathfs{L} f(\zeta)\le\int_{-\infty}^{\infty}f'\left(\zeta+\frac{1}{1+\zeta}\right)\Delta(\zeta,x)dx&=f'\left(\zeta+\frac{1}{1+\zeta}\right)\int_{-\infty}^{\infty}\Delta(\zeta,x)dx\\&=\frac{\pi}{2}\alpha e^{\alpha\left(\zeta+\frac{1}{1+\zeta}\right)}
.\eae
By Dynkin's formula

\bae\label{eq:moment_recursion}
\ev[f(Y_t)]&=\ev[f(Y_0)]+\int_0^t\ev[\mathfs{L}f(Y_s)]ds\le \ev[f(Y_0)]+\int_0^t\ev\left[\frac{\pi}{2}\alpha f( Y_s)e^{\frac{\alpha}{1+Y_s}}\right]ds\\
&\le \ev[f(Y_0)]+\frac{\alpha\pi}{2}\int_0^t\ev[ f( Y_s)]\ev\left[e^{\frac{\alpha}{1+Y_s}}\right]ds
,\eae
where in the last inequality we used the negative correlation of $e^{\alpha Y_t}$ and $e^{\frac{\alpha}{1+Y_t}}$ for $\alpha>0$. 
We can bound the second expectation by
\begin{equation}\label{eq:secondexpectation}
\ev\left[e^{\frac{\alpha}{Y_s}}\left(\ind_{H_2(s)}+\ind_{H_2(s)^c}\right)\right]\le e^{\frac{\alpha}{1+\frac{\pi}{2}s}}+\frac{c\alpha}{s^{1+1/6}}.
\end{equation}
Now let $u(t)=\ev\left[f(Y_t)\right]=\ev\left[e^{\alpha Y_t}\right]$, then by \eqref{eq:secondexpectation}

\bae\label{eq:integralineq}
u(t)\le u(0)+\frac{\pi}{2}\alpha\int_0^t u(s) e^{\frac{\alpha}{1+\frac{\pi}{2}s}} ds
.\eae
\begin{lemma}
The integral inequality \eqref{eq:integralineq} admits for all $t>\alpha$,
\bae
u(t)\le c\alpha e^{\alpha} te^{\frac{\pi}{2}\alpha t}
.\eae
\end{lemma}
\begin{proof}
By the integral Gr\"onwall's inequality 
\bae
u(t)&\le e^{\frac{\pi}{2}\alpha \int_0^t  e^{\frac{\alpha}{1+\frac{\pi}{2}r}}dr}\\
&\le e^{\frac{\pi}{2}\alpha \left[\int_0^\alpha  e^{\frac{\alpha}{1+\frac{\pi}{2}r}}dr+ \int_\alpha^t e^{\frac{\alpha}{1+\frac{\pi}{2}r}}dr \right]}\\
&\le e^{\frac{\pi}{2}\alpha^2e^\alpha } e^{\frac{\pi}{2}\alpha t +\alpha^2\log t}\\
&\le e^{\frac{\pi}{2}\alpha^2e^\alpha }  t^{\alpha^2} e^{\frac{\pi}{2}\alpha t }.
\eae
\end{proof}

Next since $M_t(0)=F_t(0)-i\frac{\pi}{2}t$, we obtain for all $\alpha>0$,
\bae
\ev\left[e^{\alpha \ima M_t(0)}\right]=\ev\left[e^{\alpha \ima F_t(0)-\frac{\pi}{2}\alpha t}\right]\le e^{\frac{\pi}{2}\alpha^2e^\alpha }  t^{\alpha^2}.
\eae
\end{proof}

By Chernoff bound,  for any $\alpha>0$, $\prob(\ima M_t(0)>\beta)\le \frac{e^{\frac{\pi}{2}\alpha^2e^\alpha } c t^{\alpha^2}}{e^{\alpha \beta}}$
and in particular $\prob(\ima M_t(0)>\beta\log t)\le e^{\frac{\pi}{2}\alpha^2e^\alpha }  t^{\alpha^2-\alpha \beta}$. Optimizing over $\alpha$, we obtain
\begin{corollary}\label{cor:tail}
$$
\prob(\ima M_t(0)>\beta\log t)\le e^{\frac{\pi \beta^2}{8}e^{\beta/2}}t^{-\frac{\beta^2}{4}}
.$$
\end{corollary}

\section{Maximal fluctuations}
To get the maximum over $x\in[0,t]$, we take small intervals and use concentration of the map and it's derivative. To bridge the gaps we use distortion theorems, but for that need to take slightly positive imaginary numbers. 

We will first bound 
\bae
\mathfs{M}_t(\zeta)=\max_{x\in[0,t]}\ima M_t(x+ i \zeta).
\eae

By Corollary \ref{cor:tail}, Theorem \ref{thm:moment} and a union bound we obtain for any $\zeta>0$
\bae\label{eq:flucdiscbound}
\prob\left(\bigcup_{j=1}^{t^{1+\gamma}}\left\{\ima M_t(j t^{-\gamma}+i\zeta)>\beta \log t\right\}\right)\le e^{\frac{\pi \beta^2}{8}e^{\beta/2}}t^{1+\gamma-\frac{\beta^2}{4}}.
\eae
The following lemma is immediate from the proof of \cite[Theorem 5.6]{berger2022growth} with $0$ replaced with a higher point $i\log t$.
\begin{lemma}\label{lem:derivmoments}
There exists a $\xi_{\ref{lem:derivmoments}}>0$ satisfying for large enough $t>0$,
$$
\ev\left[|F_t'(i\log t)|^2\right]\le \xi_{\ref{lem:derivmoments}}
.$$
\end{lemma}

By Lemma \ref{lem:derivmoments} it is immediate to get concentration bound for the derivative:
\begin{corollary}\label{cor:derbounds}
There is a $c>0$ such that for any $\gamma>6$, 
$$\prob\left(\bigcup_{j=1}^{t^{1+\gamma}}|F_t'(jt^{-\gamma}+i\log t)|>t^{2\gamma/3}\right)<\frac{\xi_{\ref{lem:derivmoments}}}{t}.$$
\end{corollary}

\begin{proof}
First by Lemma \ref{lem:derivmoments} and Markov's inequality,
$\prob\left(|F_t'(i\log t)|^2>{\kappa}\right)\le \frac{\xi_{\ref{lem:derivmoments}}}{\kappa}$. Thus for $\kappa=t^{2\gamma/3}$,
\bae
\prob\left(\bigcup_{j=1}^{t^{1+\gamma}}|F_t'(jt^{-\gamma}+i\log t)|>t^{2\gamma/3}\right)<\frac{\xi_{\ref{lem:derivmoments}} t^{1+\gamma}}{t^{4\gamma/3}}=\frac{\xi_{\ref{lem:derivmoments}}}{t^{\gamma/3-1}}.\eae
\end{proof}

We will make use of the following conformal distortion theorem, whose proof can be found in \cite[Proof of Theorem 4.3]{berger2022growth}:
\begin{theorem}[Koebe 1/4 theorem for half plane]\label{thm:koebe}
Let $f:\BH\to\BC$ be conformal, then for any $z\in\BH$ and any $w\in B(z,\ima z/2)$, 
$$
|f'(w)|\le 16 |f'(z)|
.$$
\end{theorem}
Next we present a uniform bound for points on a longer line of height $\log t$:
\begin{lemma}\label{lem:boundatlogheight} There exists a $\beta>0$ and $c>0$ such that
\bae\label{eq:log_height}
\prob\left(\max_{x\in[-t,2t]}\ima M_t(x+i\log t)>\beta \log t\right)<\frac{c}{t}.
\eae
\end{lemma}
\begin{proof}
Note that if $\{\ima M_t(j t^{-\gamma}+i\log t)\le\beta \log t\}$ and $\{|F_t'(jt^{-\gamma}+i\log t)|\le t^{2\gamma/3}\}$, then for every $x\in[j t^{-\gamma},(j+1) t^{-\gamma})$, by Theorem \ref{thm:koebe}
\bae\label{eq:koebelogheight}
|M_t'(x+i\log t)|\le 16 |M_t'(jt^{-\gamma}+i\log t)|,
\eae
and thus,
\bae
\ima M_t(x+i\log t)&\le \ima M_t(j t^{-\gamma}+i\log t)+\int_{jt^{-\gamma}}^x|M_t'(y+i\log t)|dy\\
&\le\ima M_t(j t^{-\gamma}+i\log t)+16t^{-\gamma}|M_t'(jt^{-\gamma}+i\log t)|\\
&\le 2\beta \log t.
\eae  
Since 
\bae
\left\{\max_{x\in[0,t]}\ima M_t(x+i\log t)\le 2\beta \log t\right\}&\supset \bigcap_{j=1}^{t^{1+\gamma}}\left\{\ima M_t(j t^{-\gamma}+i\zeta)\le\beta \log t\right\}\\
&\cap \bigcap_{j=1}^{t^{1+\gamma}}\left\{|F_t'(jt^{-\gamma}+i\log t)|\le t^{\gamma}\right\}
,\eae
 by \eqref{eq:flucdiscbound} and Corollary \ref{cor:derbounds},
 \bae
\prob\left(\max_{x\in[-t,2t]}\ima M_t(x+i\log t)>2\beta \log t\right)<e^{\frac{\pi \beta^2}{8}e^{\beta/2}}t^{1+\gamma-\frac{\beta^2}{4}}+\frac{c}{t}.
\eae
Now the claim follows by choosing $\gamma>6$ and $\beta$ large enough. 
\end{proof}
We will also need a uniform bound on the real fluctuations:
\begin{lemma}
There exists a $c>0$ such that,
\bae\label{eq:real_log}
\prob\left(\max_{x\in[-t,2t]}|\rea M_t(x+i\log t)|> t^{2/3} \right)<\frac{c \log t}{t^{1/6}}.
\eae
\end{lemma}
\begin{proof}
By Theorem \ref{thm:var}, for any ${\tilde{\gamma}}<1/6$
\bae
\prob\left(\bigcup_{j=1}^{t^{1+{\tilde{\gamma}}}}|\rea M_t(jt^{-{\tilde{\gamma}}}+i\log t)|>t^{2/3}\right)<\frac{\log t \cdot t^{1+{\tilde{\gamma}}}}{t^{4/3}}=\frac{\log t }{t^{1/3-{\tilde{\gamma}}}}\le\frac{c \log t}{t^{1/6}}.
\eae
Moreover,
\bae
\prob\left(\bigcup_{j=1}^{t^{1+1/6}}|F_t'(jt^{-\gamma}+i\log t)|>t^{5/6}\right)<\frac{\xi_{\ref{lem:derivmoments}} t^{1+1/6}}{t^{10/6}}=\frac{\xi_{\ref{lem:derivmoments}}}{t^{1/2}}.
\eae

Note that if $\{|\rea M_t(j t^{-{\tilde{\gamma}}}+i\log t)|\le t^{\frac{2}{3}}\}$ and $\{|F_t'(jt^{-{\tilde{\gamma}}}+i\log t)|\le t^{5/6}\}$, then for every $x\in[j t^{-{\tilde{\gamma}}},(j+1) t^{-{\tilde{\gamma}}})$,
\bae
|\rea M_t(x+i\log t)|\le |\rea M_t(j t^{-{\tilde{\gamma}}}+i\log t)+t^{-1/6}|F_t'(jt^{-{\tilde{\gamma}}}+i\log t)|\le t^{2/3}.
\eae 

\end{proof}


\begin{proof}[Proof of Theorem \ref{thm:max_fluc}]
We use planarity to show that $F_t([0,t])$ can't overpass $F_t(i\log t+[-t,2t])$. 

Using the notations of \cite{berger2022growth}, let $\nu_{F_{2t}(\frac{3}{2}t)}^{2t}:[0,1]\to F_t(\partial\BH)$ be the unique path connecting $F_{2t}\left(\frac{3}{2}t\right)$ to $\partial \BH$ in $F_{2t}(\partial \BH)$. Similarly define  $\nu_{F_{2t}(-\frac{1}{2}t)}^{2t}$. By \cite[Theorems 7.2+7.7]{berger2022growth} and \cite[Proposition 2.4.5]{lawler2010random},
\bae\label{eq:pathdown}
&\prob\left(\left\{F_t^{-1}\left(\nu_{F_{2t}(\frac{3}{2}t)}^{2t}\right)\cap [0,t]=\emptyset\right\}\cap\left\{\ima F_{2t}(\frac{3}{2}t)>\frac{\pi}{2}t+3\beta \log t\right\}\right)&>1- c e^{-c t^{1/6}},\\
&\prob\left(\left\{F_t^{-1}\left(\nu_{F_{2t}(-\frac{1}{2}t)}^{2t}\right)\cap [0,t]=\emptyset\right\}\cap\left\{\ima F_{2t}(-\frac{1}{2}t)>\frac{\pi}{2}t+3\beta \log t\right\}\right)&>1- c e^{-c t^{1/6}}.
\eae
Denote $\mathfs{V}$ the intersection of the events in \eqref{eq:log_height}, \eqref{eq:real_log} and \eqref{eq:pathdown}. Let 
\bae 
\tilde{t}_l=\inf\left\{s\in[0,1]:\nu_{F_{2t}(-\frac{1}{2}t)}^{2t}(s)\in F_t(i\log t+[-t,2t])\right\},\\
\tilde{t}_r=\inf\left\{s\in[0,1]:\nu_{F_{2t}(\frac{3}{2}t)}^{2t}(s)\in F_t(i\log t+[-t,2t])\right\},
\eae
Under the event $\mathfs{V}$ we have that $F_t([0,t])$ is contained in the area bounded by 
\bae\label{eq:bounding set}
\nu_{F_{2t}(-\frac{1}{2}t)}^{2t}([0,\tilde{t}_l]) \cup F_t(i\log t+[-t,2t])\cup \nu_{F_{2t}\left(\frac{3}{2}t\right)}^{2t}([0,\tilde{t}_r]),
\eae
whose imaginary part is bounded by $\frac{\pi}{2}t+2\beta \log t$, under the same event (See Figure \ref{figure:4}).

\begin{figure}

\begin{tikzpicture}
    \draw[thick] (-2, 0) -- (4, 0);
    
    \foreach \x/\label in {-2/$-t$, 0/$0$, 2/$t$, 4/$2t$} {
        \draw[thick] (\x, 0.1) -- (\x, -0.1); 
        \node[below] at (\x, 0) {\label}; 
    }
    
    \draw[thick, blue] (-2, 0.5) -- (4, 0.5);
    \node[right, blue] at (4, 0.5) {{\footnotesize $i \log t$}};
    
    \draw[thick, red] (-2, 3.5) -- (4, 3.5);
    \node[right, red] at (4, 3.5) {{\footnotesize $i(\frac{\pi}{2}t+2\beta \log t)+\BR$}};
    
    \draw[thick, blue] (-2, 3) -- (-1.5,3.3)-- (0,2.9)-- (2,3.4)--(4, 3);
    \node[right, blue] at (4, 3) {{ \footnotesize $F_t(i\log t+[-t,2t])$}};
    
    

    \coordinate (root) at (-1, 0); 
    \draw[thick, green!70!black] (root) -- ++(0.5, 1.5); 
    
    \draw[thick, green!70!black] (-0.5, 1.5) -- ++(-0.5, 1.0);


    \node[above, green!70!black] at (4, 1.5) { {\footnotesize $\nu_{F_t(\frac{3}{2}t)}^t$}};
    
     \coordinate (root) at (3, 0); 
    \draw[thick, green!70!black] (root) -- ++(0.5, 1.5); 
    
    \draw[thick, green!70!black] (root) -- ++(0.5, 1.5)--++(-0.5, 1.25)--++(-0.05,1.7);

    \draw[thick, green!70!black] (-1.0, 2.5) -- ++(0.25, 0.25)--++(0.3,2);

    \node[above, green!70!black] at (-1.5, 1.5) {{\footnotesize $\nu_{F_t(-\frac{1}{2}t)}^t$}};
     \node[above, black!70!black] at (2.3, 1.5) {{\footnotesize $F_t([0,t])$}};

   \node[right, blue] at (3, 3.3) {$z_r$};
   \node[draw, fill=blue, circle, inner sep=0pt, minimum size=3pt] at (3,3.2) {};
   
      \node[right, blue] at (-1, 3.3) {$z_l$};
   \node[draw, fill=blue, circle, inner sep=0pt, minimum size=3pt] at (-0.7,3.1) {};
   

     

    \coordinate (root) at (0.5, 0); 
    \draw[thick, black!70!black] (root) -- ++(0.5, 1); 
    \draw[thick, black!70!black] (root) -- ++(-0.55, 1.1); 
    
    \draw[thick, black!70!black] (0, 1)-- ++(0.25, 0.5);
    \draw[thick, black!70!black] (1, 1) -- ++(0.5, 1.0);
        \draw[thick, black!70!black] (1, 1) -- ++(-0.5, 0.7);

    \draw[thick, black!70!black] (0.5, 1.7) -- ++(-0.25, 0.5);
    \draw[thick, black!70!black] (0.5, 1.7) -- ++(0.25, 0.6);

        \coordinate (root) at (1.5, 0); 
    \draw[thick, black!70!black] (root) -- ++(0.8, 1); 
    \draw[thick, black!70!black] (root) -- ++(-0.35, 0.7); 

    \draw[thick, black!70!black] (2.3, 1) -- ++(-0.25, 0.5);
    \draw[thick, black!70!black] (2.3, 1)  -- ++(0.25, 0.6);

\end{tikzpicture}
\caption{Illustration of bounding set \eqref{eq:bounding set}, with $z_l=\nu_{F_{2t}(-\frac{1}{2}t)}^{2t}(\tilde{t}_l)$ and $z_r=\nu_{F_{2t}(\frac{3}{2}t)}^{2t}(\tilde{t}_r)$.  \label{figure:4}}
\end{figure}
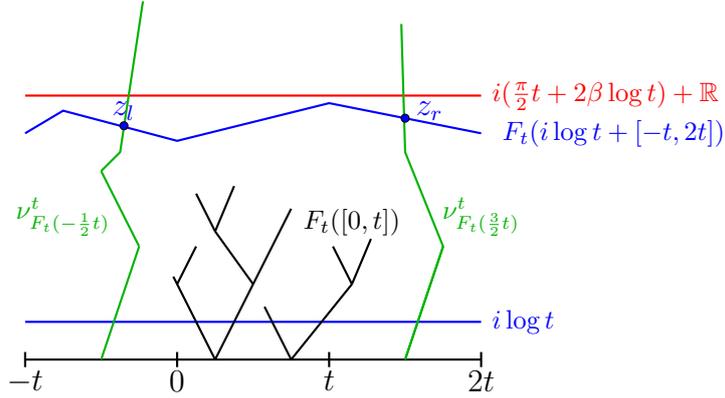

\end{proof}
\appendix
\section{New Lemma 3.3}
\begin{lemma}\label{lem:slit_int_bound}
There exists a $c_{\ref{lem:slit_int_bound}}>0$ and a $\tilde{y}>0$, such that for any $y>\tilde{y}$
\bae
\Big|\int_{-\infty}^{\infty}|\slit_x(iy)-iy|^2dx-\min\{1,\frac{\pi}{4y}\}\Big|<\frac{c_{\ref{lem:slit_int_bound}}}{y^3}.
\eae 
\end{lemma}
\begin{proof}
Using Taylor expansion of $\sqrt{\cdot}$ at $y\to\infty$
\bae
\Big|\slit_x(iy)-iy\Big|^2&=\Big|\sqrt{(iy-x)^2-1}-(iy-x)\Big|^2=\Big|\frac{1}{2(iy-x)}+O\left(\frac{1}{(iy-x)^3}\right)\Big|^2\\
&=\frac{1}{4(x^2+y^2)}+O\left(\frac{1}{(x^2+y^2)^2}\right)
.\eae
Thus,
\bae
\int_{-\infty}^{\infty}|\slit_x(iy)-iy|^2dx=\min\{1,\frac{\pi}{4y}\}+O\left(\frac{1}{y^3}\right)
\eae
\end{proof}

\section*{Acknoledgements}
This work is partially supported by DFG grant 5010383. The authors would like to thank Ofer Zeitouni for fruitful discussions. 

	\bibliography{career}

\begin{thebibliography}{MPZ22}

\bibitem[AP17]{antunovic2017stationary}
Ton{\'c}i Antunovi{\'c} and Eviatar~B Procaccia.
\newblock Stationary eden model on cayley graphs.
\newblock {\em The Annals of Applied Probability}, pages 517--549, 2017.

\bibitem[BKP14]{berger2014stretched}
Noam Berger, Jacob~J Kagan, and Eviatar~B Procaccia.
\newblock Stretched idla.
\newblock {\em ALEA Lat. Am. J. Probab. Math. Stat}, 11(1):471--481, 2014.

\bibitem[BPT22]{berger2022growth}
Noam Berger, Eviatar~B Procaccia, and Amanda Turner.
\newblock Growth of stationary hastings--levitov.
\newblock {\em The Annals of Applied Probability}, 32(5):3331--3360, 2022.

\bibitem[HL98]{hastings1998laplacian}
Matthew~B Hastings and Leonid~S Levitov.
\newblock Laplacian growth as one-dimensional turbulence.
\newblock {\em Physica D: Nonlinear Phenomena}, 116(1-2):244--252, 1998.

\bibitem[LL10]{lawler2010random}
G.~F. Lawler and V.~Limic.
\newblock {\em {Random walk: a modern introduction}}.
\newblock Cambridge Univ Pr, 2010.

\bibitem[Mea83]{meakin1983diffusion}
Paul Meakin.
\newblock Diffusion-controlled deposition on fibers and surfaces.
\newblock {\em Physical Review A}, 27(5):2616, 1983.

\bibitem[MPZ22]{mu2022scaling}
Yingxin Mu, Eviatar Procaccia, and Yuan Zhang.
\newblock Scaling limit of dla on a long line segment.
\newblock {\em Transactions of the American Mathematical Society},
  375(12):8769--8806, 2022.

\bibitem[NT12]{norris2012hastings}
J.~Norris and A.~Turner.
\newblock Hastings--levitov aggregation in the small-particle limit.
\newblock {\em Communications in Mathematical Physics}, pages 1--33, 2012.

\bibitem[PP21]{procaccia2021dimension}
Eviatar~B Procaccia and Itamar Procaccia.
\newblock Dimension of diffusion-limited aggregates grown on a line.
\newblock {\em Physical Review E}, 103(2):L020101, 2021.

\bibitem[PYZ20]{procaccia2020stationary}
Eviatar~B Procaccia, Jiayan Ye, and Yuan Zhang.
\newblock Stationary dla is well defined.
\newblock {\em Journal of Statistical Physics}, 181:1089--1111, 2020.

\bibitem[PZ19]{procaccia2019stationary}
Eviatar~B Procaccia and Yuan Zhang.
\newblock Stationary harmonic measure and dla in the upper half plane.
\newblock {\em Journal of Statistical Physics}, 176(4):946--980, 2019.

\bibitem[Sil17]{silvestri2017fluctuation}
Vittoria Silvestri.
\newblock Fluctuation results for hastings--levitov planar growth.
\newblock {\em Probability Theory and Related Fields}, 167:417--460, 2017.

\end{thebibliography}
	\bibliographystyle{alpha}

\end{document}